\documentclass[12pt,a4paper]{article}

\usepackage{amsmath,euscript}
\usepackage{amssymb}
\usepackage{amsthm}
\usepackage{enumerate}
\usepackage{amsfonts}
\usepackage{comment}
\usepackage{amscd}

\newtheorem{theorem}{Theorem}[section]
\newtheorem{lemma}[theorem]{ Lemma}

\DeclareMathOperator{\GKdim}{GKdim}

\title{Nil algebras with restricted growth}

\author{T H Lenagan, Agata Smoktunowicz
\thanks{\,The research of the second author was
supported by Grant No. EPSRC
 EP/D071674/1.
 }
~~\\and Alexander Young\thanks{\,The research of the third author was
partially supported by the United States National
Science Foundation.}}

\date{ }

\begin{document}

\maketitle

\begin{abstract}
It is shown that over an arbitrary countable field, there exists a
finitely generated algebra that is nil, infinite dimensional, and
has Gelfand-Kirillov dimension at most three.
\end{abstract}




\section*{Introduction}

In 1902 William Burnside asked the following question which later became known
as the Burnside Problem: does a finitely generated group whose elements all
have finite order need to be finite? An analogous problem for algebras is the
Kurosh Problem: if $A$ is a finitely generated algebra over a field $K$, and
every element of $A$ is algebraic over $K$, does it follow that $A$ is finite
dimensional over $K$? A special case of the Kurosh Problem, sometimes known as
Levitski's Problem, concerns nil algebras: if $A$ is a finitely generated
algebra over a field $K$ and every element of $A$ is nilpotent, is $A$ finite
dimensional over $K$?

 The seminal work of Golod and Shafarevich, \cite{golod, gs}, in 1964, showed
that the answer to these famous problems was negative. Their method entailed
the construction of a finitely generated nil algebra $A$ which was infinite
dimensional and then from this algebra the counterexample to the Burnside
Problem arises by considering a group whose elements are of the form $1+n$,
for a particular nil algebra $A$ and some $n\in A$.

The groups and the algebras constructed by the Golod-Shafarevich method have
exponential growth. Much later, Gromov \cite{gromov} proved that under the
assumption that the group has polynomial growth, the answer to the Burnside
Problem is positive. In fact, he proved that a finitely generated group with
polynomial growth has a nilpotent normal subgroup of finite index. As a
consequence, if a finitely-generated group has polynomial growth and each
element has finite order then the group is finite.

Golod and Shafarevich's work together with Gromov's result naturally raises
the question as to whether a finitely generated nil algebra with polynomial
growth is of necessity finite dimensional, \cite{small, ufn}. Suprisingly,
this is not the case: in \cite{ls} Lenagan and Smoktunowicz constructed, over
any countable field, an infinite dimensional finitely generated nil algebra
with Gelfand-Kirillov dimension at most $20$. This result raises the following
question: what is the minimal rate of growth for a finitely generated infinite
dimensional nil algebra? In this paper, we make progress on this latter
question: by refining the methods of \cite{ls}, we construct, over any
countable field, an infinite dimensional finitely generated nil algebra with
Gelfand-Kirillov dimension at most $3$. (In fact, our algebra requires only
two generators.)




\section{Notation}

In what follows, $K$ will be a countable field and $A$ will be the free
$K$-algebra in two non-commuting indeterminates $x$ and $y$. The
set of monomials in $x$, $y$ is denoted by $M$, and $M(n)$
denotes the set of monomials of degree $n$, for each $n\geq 0$.
Thus, $M(0)=\{1\}$ and for $n\geq1$ the elements in $M(n)$ are of
the form $x_1...x_n$, where $x_i\in \{x, y\}$. The
$K$-subspace of $A$ spanned by $M(n)$ will be denoted by $H(n)$
and elements of $H(n)$ will be called {\em homogenous polynomials
of degree $n$}.  The {\em degree}, $\deg f$, of any $f \in A$,
is the least $d \geq 0$ such that $f \in H(0) + ... + H(d)$. Any
$f\in A$ can be uniquely presented in the form $f=f_0+f_1+...
+f_d$, where each $f_i\in H(i)$. The elements $f_{i}$ are the {\em
homogeneous components} of $f$.  A right ideal $I$ of $A$ is
{\em homogeneous} if for every $f\in I$ all homogeneous components of
$f$ are in $I$.  If $V$ is a linear space over $K$, then $\dim_K V$
denotes the dimension of $V$ over $K$.  The Gelfand-Kirillov dimension
of an algebra $R$ is denoted by $\GKdim(R)$. For elementary properties
of Gelfand-Kirillov dimension we refer to \cite{kl}.

For any real number $k$, define $\lfloor k \rfloor$ to be the largest
integer not exceeding $k$.

Throughout the paper,  $\bar{A}$ will denote
the subalgebra of $A$ consisting of polynomials
with constant term equal to zero.

Assume that all logarithms in this paper are of base 2.

The aim of this paper is to
present an algebra with the desired properties
in the form $\bar{A}/E$ for a suitable ideal $E$.
First, we will construct a sequence of linear spaces $U(2^n)$, and
then set $E$ to be the largest subset that for all $n\geq 0$,
$AEA\cap H(2^n)\subseteq U(2^n)$. As the sets $U(2^n)$ will be very large
in dimension ($\dim_K U(2^n) + 2=\dim_K H(2^n)$ for most $n$) and behave
like an ideal (that is, $H(2^n)U(2^n)+U(2^n)H(2^n)\subseteq U(2^{n+1})$),
the ideal $E$ will be very large and hence $\GKdim \bar{A}/E$ will be small.
To guarantee that the algebra $\bar{A}/E$ is nil we allow the
sets $U(2^n)$ to
have a bigger co-dimension at some sparse places.




\section{Enumerating elements}

We start with the following Lemma.

\begin{lemma}\label{enumerate}
    Let $K$ be a countable field, and let $\bar{A}$ be as above.
    Then there exists a subset $Z\subseteq \mathbb{N}$, with all
    $i\in Z$ being greater than or equal to $5$, and an enumeration
    $\{f_i\}_{i\in Z}$ of $\bar {A}$ such that
    $\lfloor \log i \rfloor>6^{6\deg f_i}$.
    Moreover, the set $Z$ has the following property: if $i>j$ and
    $i,j\in Z$ then $i > 2^{2^{2^{2j}}}$.
\end{lemma}

\begin{proof}
    As $\bar{A}$ is a finitely-generated algebra over a countable
    field, it is itself countable. Let $\bar{A} = \{g_1, g_2, ...\}$ be
    an arbitrary enumeration. We now inductively define an increasing function
    $\theta: \mathbb{N} \rightarrow \mathbb{N}$ as follows: first set
    $\theta(1):= \min \{i \in \mathbb{N}| i>4, \lfloor \log i \rfloor>
    6^{6\deg g_1}\}$.

    As an inductive hypothesis, suppose that $\theta$ is defined over $\{1,
    ..., n\}$ such that $\lfloor \log(\theta(i))\rfloor > 6^{6\deg g_i}$ for
    each $i\leq n$. Then set $\theta(n+1) = \min\{s \in \mathbb{N}| \lfloor
    \log s \rfloor > 6^{6\deg g_{n+1}}, s>2^{2^{2^{2\theta(n)}}}\}$. If we now rename
    the elements of $\bar{A}$ by setting $f_{\theta(s)} = g_s$ then we have a
    listing of the elements of $\bar{A}$ with the required properties.
\end{proof}

\begin{theorem}\label{existFi}
    Let $Z$ and $\{f_{i}\}_{i\in Z}$ be as in Lemma~\ref{enumerate}. Let $i\in
    Z$, and let $I$ be the two-sided ideal generated by $f_i^{10w_i}$ where
    $w_i=4\cdot2^{2^i-\lfloor \log i\rfloor}$. There is a linear $K$-space
    $F_i \subseteq H(2^{2^{i}-\lfloor \log i \rfloor})$ such that $I\subseteq
    \sum_{k=0}^\infty H(k(2^{2^i-\lfloor \log i \rfloor})) F_i A$ and
    $\dim_{K}(F_i)<2^{2^i}-2$.
\end{theorem}

\begin{proof}
 Note that $6^{6\deg (f_i)}<\lfloor \log i \rfloor$ by Lemma~ \ref{enumerate}.
 Apply \cite[Theorem 2]{smok} with $f=f_{i}$, $r=2^{2^i-\lfloor \log
 i\rfloor}$, $w=w_i=4\cdot2^{2^i-\lfloor \log i\rfloor}$, and put $F_i=span_K
 F$, where $F$ is the corresponding set $F$ of the conclusion of Lemma~
 \ref{enumerate}. Note that these choices of $f,r,w$ satisfy the hypotheses of
 \cite[Theorem 2]{smok}. Although the algebra $A$ in \cite[Theorem 2]{smok} is
 generated by three elements not by two, this does not influence the proof.
\end{proof}



\section{Definition of $U(2^n)$ and $V(2^n)$}\label{defuv}

In this section we will define a set $U$ with the properties mentioned in
the introduction. In order to construct $U$ we will first construct the
sets $U(2^n)=H(2^n)\cap U$, for $n=1, 2, \dots$. In the next
section we will construct the ideal $E$ by requiring that $r\in E$ if
$prq\in U$ for all $p, q\in A$.

For each $i \in Z$, set $S_i = [2^i-i-\lfloor \log i \rfloor, 2^i-\lfloor \log
i \rfloor-1]$, and set $S=\bigcup_{i\in Z} S_i$. Note that the $S_i$
are pairwise disjoint.

\begin{theorem} \label{8props}
    Let $Z$, $F_{i}$  be as in Theorem~\ref{existFi}.
    There are $K$-linear subspaces $U(2^{n})$ and $V(2^{n})$ of $H(2^{n})$
    such that for all $n>0$:
    \begin{enumerate}
        \item $\dim_{K}V({2^{n}})=2$ if $n\notin S$;

        \item $\dim_{K}V(2^{2^{i}-i-\lfloor log(i)\rfloor+j})=2^{2^{j}}$, for
        all $1<i\in Z$ and all $0\leq j\leq i-1$;

        \item $V(2^{n})$ is spanned by monomials;

        \item $F_{i}\subseteq U(2^{2^{i}-\lfloor log(i)\rfloor})$ for every
        $i\in Z$;

        \item $V(2^{n})\oplus U(2^{n})=H(2^{n})$;
        \item $H(2^{n})U(2^{n})+U(2^{n})H(2^{n})\subseteq U(2^{n+1})$;
        \item $V(2^{n+1})\subseteq V(2^{n})V(2^{n})$;
        \item  if $n\notin S$ then there are monomials $m_{1},
            m_{2}\in V(2^{n})$ such that $V(2^n)=Km_{1}+Km_{2}$ and
            $m_{2}H(2^{n})\subseteq U(2^{n+1})$.

    \end{enumerate}

\end{theorem}
\begin{proof}
    The proof of properties (1) to (7) is very similar to the proof of
    \cite[Theorem~3]{ls} and the proof of property (8) is similar to the proof
    of \cite[Theorem~10(8)]{as}. We construct the sets $U(2^n)$ and $V(2^n)$
    inductively. Set $V(2^{0})= Kx+Ky$ and $U(2^{0})=0$. Assume that
    we have defined $V(2^{m})$ and $U(2^{m})$ for $m\leq n$ in such a way that
    conditions 1-5 hold for all $m\leq n$ and conditions 6,7 and 8 hold for
    all $m<n$. Then we define $V(2^{n+1})$ and $U(2^{n+1})$ inductively, in
    the following way. Consider the three cases

\begin{itemize}
        \item[1.] $n\in S$ and $n+1\in S$.
        \item[2.] $n\notin S$.
        \item[3.] $n\in S$ and $n+1\notin S$.
    \end{itemize}

    {\bf Case 1.} If $n\in S$ and $n+1\in S$, define $U(2^{n+1}) =
    H(2^n)U(2^n) + U(2^n)H(2^n)$ and $V(2^{n+1})=V(2^{n})V(2^{n})$. Conditions
    6, 7 certainly hold. If, by induction, conditions 5 and 3 hold for
    $U(2^n)$ and $V(2^n)$, they hold for $U(2^{n+1})$ and $V(2^{n+1})$ as
    well. Moreover, $\dim_{K}V(2^n)=(\dim_{K}V(2^n))^2$, inductively
    satisfying condition 2.

    {\bf Case 2.} Suppose that $n\notin S$. Then $\dim_{K}V(2^{n})=2$, as is
    generated by monomials, by the inductive hypothesis. Let $m_1, m_2$ be the
    distinct monomials that generate $V(2^n)$. Then $V(2^n)V(2^n)=Km_1 m_1 +
    Km_1 m_2+Km_2 m_1 + Km_2 m_2$. Set $V(2^{n+1})=Km_1 m_1+Km_1 m_2$, so that
    conditions 1, 3, 7 and 8 hold.

    Set $U(2^{n+1}) = H(2^n)U(2^n) + U(2^n)H(2^n) + m_2 V(2^n)$.
    Using this definition, condition 6 holds and
    \begin{eqnarray*}
    \lefteqn{H(2^{n+1}) = H(2^n)H(2^n)}\\
    &=&U(2^n)U(2^n) \oplus U(2^n)V(2^n) \oplus V(2^n)U(2^n)
        \oplus m_1 V(2^n) \oplus m_2 V(2^n)\\
    &=&U(2^{n+1}) \oplus V(2^{n+1})
    \end{eqnarray*}
    Thus condition 5 holds.

    {\bf Case 3.} Suppose that $n\in S$ while $n+1\notin S$. Then
    $n=2^i-\lfloor log(i)\rfloor-1$ for some $i\in Z$. By induction on
    condition 2, $\dim_K V(2^n)= \dim_K V(2^{2^i-i-\lfloor
    log(i)\rfloor+i-1})=2^{2^{i-1}}$, and $\dim_K
    V(2^n)V(2^n)=2^{2^{i-1}}2^{2^{i-1}}=2^{2^i}$. By induction on condition 5,
    $$
    H(2^{n+1}) = U(2^n)U(2^n) \oplus U(2^n)V(2^n)
    \oplus V(2^n)U(2^n) \oplus    V(2^n)V(2^n).
    $$
    We know that $F_i$ has a basis $\{f_{1}, \dots ,f_{s}\}$
    for some $f_{1},\dots , f_{s}\in H(2^{2^{i}}-\lfloor log(i)\rfloor)$ and
    $s<2^{2^{i}}-2$. Each $f_j$ can be uniquely decomposed into $\bar{f}_j+g_j$
    with $\bar {f}_j \in V(2^{n})V(2^{n})$ and $g_{j}\in V(2^{n})U(2^{n})
    +U(2^{n})U(2^{n})+U(2^{n})V(2^{n})$. Let $P$ the subspace spanned by
    $\bar{f}_1, ..., \bar{f}_s$.

    Since $\dim_K P \leq s = \dim F_i < 2^{2^i} - 2 < \dim_K V(2^n)V(2^n) - 2$, there must
    exist at least two monomials $m_1, m_2 \in V(2^n)V(2^n)$ such that the space $K m_1 + K m_2$
    is disjoint from $P$.  Define $V(2^{n+1})$ as this space; This satisfies conditions 1, 3 and 7.

    As $P$ is disjoint from $K m_1 + K m_2$, there must exist a space $Q \supseteq P$ such that
    $V(2^n)V(2^n) = Q \oplus (K m_1, K m_2)$.  Set:
        $$U(2^{n+1}) = U(2^n)U(2^n) + U(2^n)V(2^n) + V(2^n)U(2^n) + Q$$
    This immediately satisfies conditions 5 and 6.  Since each polynomial $f_i = g_i + \bar{f}_i \in
    U(2^{n+1})$, it satisfies condition 4 as well.
\end{proof}

Before continuing, a helpful lemma concerning of $U(2^n)$ should be mentioned.

\begin{lemma}\label{Ustack}
    For any $m \geq n$, and any $0 \leq k < 2^{m-n}$,
    \[
    H(k
    2^n)U(2^n)H((2^{m-n}-k-1) 2^n) \subseteq U(2^m).
    \]
\end{lemma}

\begin{proof}
    If $m = n$, then $k = 0$ and the equation is trivially true.
    Using induction, assume the theorem holds true for some $m \geq n$.
    When  $0 \leq k < 2^{m-n}$
    \begin{eqnarray*}
    \lefteqn{H(k 2^n)U(2^n)H((2^{m+1-n}-k-1) 2^n) =}\\
    &&H(k 2^n)U(2^n)H((2^{m-n}-k-1) 2^n) H(2^m)\subseteq
    U(2^m)H(2^m) \subseteq U(2^{m+1}),
    \end{eqnarray*}
    and when
     $2^{m-n} \leq k < 2^{m+1-n}$
    \begin{eqnarray*}
    \lefteqn{H(k 2^n)U(2^n)H((2^{m+1-n}-k-1) 2^n) =}\\
    &&H(2^m) H((k-2^{m-n}) 2^n)U(2^n)H((2^{m+1-n}-k-1) 2^n)\\
    &\subseteq&
    H(2^m)U(2^m) \subseteq U(2^{m+1}),
    \end{eqnarray*}
    as required.
\end{proof}

Another way of stating Lemma~\ref{Ustack} is that, given any product of the
form $H(i 2^n)U(2^n)H(j 2^n)$, if the sum of the three arguments $i 2^n + 2^n
+ j 2^n$ is a power of $2$ then
$H(i 2^n)U(2^n)H(j 2^n) \subseteq U(i 2^n + 2^n +
j 2^n)$.



\section{A finitely generated infinite dimensional nil algebra}

A graded subspace $E\subseteq \bar{A}$ is formed by defining its
homogeneous subspace
$E(n)$ to be the set of elements $r\in H(n)$ such that if $2^{m}\leq
n<2^{m+1}$ then for all $0\leq j\leq 2^{m+2}-n$,
    $$H(j)rH(2^{m+2}-j-n)\subseteq
    U(2^{m+1})H(2^{m+1})+ H(2^{m+1})U(2^{m+1})$$
Now, define $E:=E(1)+E(2)+ ... $.

\begin{theorem}\label{AbaroverEnil}
    The subset $E$ is an ideal in $\bar{A}$. Moreover $\bar{A}/E$ is a nil
    algebra and is
    infinite dimensional over $K$.
\end{theorem}
\begin{proof}
    The set $E$ is shown to be an ideal in \cite[Theorem 5]{ls}, and
    \cite[Theorems 14,15]{ls} prove that $\bar{A}/E$ is both nil
    and
    infinite dimensional over $K$.
    No changes to these proofs
    need to be made to apply to our example,
    and so the proofs are not repeated here.
\end{proof}



\section{The subspaces $R$, $S$, $Q$, $W$}
The key to computing the Gelfand-Kirillov dimension of the algebra
$\bar{A}/E$ is to use a
collection of subspaces $R,S, Q, W$ with the following properties:
if $n>0$, $2^m\leq n < 2^{m+1}$ then
        \begin{align*}  R(n)H(2^{m+1}-n)\subseteq U(2^{m+1})&\qquad
    H(2^{m+1}-n)S(n)\subseteq U(2^{m+1})\\
    H(n)=R(n)\oplus Q(n)&\qquad H(n)=S(n)\oplus W(n)
    \end{align*}
It then follows from Theorem~\ref{Ustack} that for any $k>n$,
$R(n)H(2^k-n)\subseteq U(2^k)$ and $H(2^k-n)S(n)\subseteq
U(2^k)$.

The existence of suitable such subspaces is established in the next section.
Once this has been acheived, the following theorem is available to help
calculate the Gelfand-Kirillov dimension of $\bar{A}/E$. 
(In this theorem we take $R(0)=S(0)=U(0)=0$ and
$V(0)=Q(0)=W(0)=K$.)

\begin{theorem}\label{totalsize}
    For every $n \in \mathbb{N}$,
    $$\bigcap_{k=0}^{n}S(n-k)H(k)+H(n-k)R(k)\subseteq E(n)$$
Moreover,
    $$\dim(H(n)/E(n))\leq \sum_{k=0}^{n}\dim(W(n-k))\dim(Q(k))$$
\end{theorem}

\begin{proof}
    The proof of the first claim is very similar to the proof
    of \cite[Theorem~9]{ls} and so is  omitted.
    Notice that,
            \begin{eqnarray*}
            H(n)&=&(S(n-k)\oplus W(n-k))(R(k)\oplus Q(k))\\
            &=&(S(n-k)H(k)+H(n-k)R(k)) \oplus W(n-k)Q(k).
            \end{eqnarray*}
    Therefore,
    \begin{eqnarray*} \dim E(n)&\geq&
    \dim\left(\bigcap_{k=0}^{n}(S(n-k)H(k)+H(n-k)R(k))\right)\\
    &\geq&\dim(H(n))-\sum_{k=0}^{n}\dim(W(n-k)Q(k))
    \end{eqnarray*}
    and so
    $$\dim(H(n)/E(n))\leq
    \sum_{k=0}^{n}\dim(W(n-k))Q(k)),$$
    as required.
\end{proof}



\section{A sufficiently small $Q$ and $W$}

In order to define $R$, $S$, $Q$ and $W$, begin with
$R(1)=S(1)=U(1)$, $Q(1)=W(1)=V(1)$. Given any natural number $j$
with  $2^{m}\leq j< 2^{m+1}$, define
    $$R(j)=\{r\in H(j): rH(2^{m+1}-j)\subseteq U(2^{m+1})\}$$ and
    $$S(j)=\{r\in H(j): H(2^{m+1}-j)r\subseteq U(2^{m+1})\}.$$

\begin{theorem}\label{Wsize}
    Let $j$ be a natural number.
    Write $j$ in  binary form as
        $$j=2^{p_0}+2^{p_1}+...+2^{p_n}$$
    with $0\leq p_0<p_1<... <p_n$. Then there is a $K$-linear space
    $W(j)\subseteq H(j)$ such that $W(j)\oplus S(j)=H(j)$ and
        $$W(j)\subseteq V(2^{p_0})...V(2^{p_n})=\prod_{i=0}^n V(2^{p_i})$$
\end{theorem}

\begin{proof}
    By Theorem \ref{8props}(5), $H(2^{p_i})=U(2^{p_i})\oplus V(2^{p_i})$ for
    $i=1,2, ..., n$. Hence, $H(j)=\prod_{i=0}^n (U(2^{p_i}) \oplus
    V(2^{p_i}))$, and
    $$H(j) = \left(\displaystyle\sum_{i=0}^n
    H(2^{p_0}+...+2^{p_{i-1}})U(2^{p_i})H(2^{p_{i+1}}+...+2^{p_n}) \right)
    \oplus \displaystyle\prod_{i=0}^n V(2^{p_i})$$ Define $T_{p_i}(j)$ as
    $H(2^{p_0}+...+2^{p_{i-1}})U(2^{p_i})H(2^{p_{i+1}}+...+2^{p_n})$, so that
    $H(j) = (\sum_{i=0}^n T_{p_i}(j)) \oplus \prod_{i=0}^n V(2^{p_i})$.

    Now, from the definition of $T_{p_i}(j)$, we obtain
        \begin{eqnarray*}
        \lefteqn{H(2^{p_n+1}-j)T_{p_i}(j) =
        H(2^{p_n+1}-(2^{p_i}+...+2^{p_n}))U(2^{p_i})H(2^{p_{i+1}}+...+2^{p_n})
        =}\\
        &\hspace{-1ex}&
        H\left((2^{p_n+1-p_i}-(2^0+...+2^{p_n-p_i}))2^{p_i}\right)U(2^{p_i})
        H\left((2^{p_{i+1}-p_i}+...+2^{p_n-p_i})2^{p_i}\right)
        \end{eqnarray*}
    It follows from Lemma~\ref{Ustack} that $H(2^{p_n+1}-j)T_{p_i}(j)
    \subseteq U(2^{p_n+1})$; so that each $T_{p_i}(j)
    \subseteq S(j)$. Thus, there must exist some $W(j) \subseteq
    \prod_{i=0}^n V(2^{p_i})$ such that $S(j) \oplus W(j) = H(j)$. To see this
    more clearly, choose a basis of $(\prod_{i=0}^n V(2^{p_i}) + S(j))/S(j)$,
    pull this basis
    back to elements in $\prod_{i=0}^n V(2^{p_i})$, and let
    $W(j)$ be the subspace generated by that basis.
\end{proof}

Next, sets $N(2^{i})$ are defined in a similar way to the procedure
used in  \cite{as}.

Let $i\notin S$. Then, by Theorem~\ref{8props}(8), each $V(2^i)$
is generated by
two monomials $m_{1,i}$ and $m_{2,i}$, with $m_{2,i}H(2^i)\subseteq
U(2^{i+1})$. Define $N(2^i)=Km_{1,i}$, and $M(2^i)=U(2^i)+Km_{2,i}$. In the
case where $i\in S$, simply set $N(2^i)=V(2^i)$, $M(2^i)=U(2^i)$. Observe that
for every $i$, $N(2^i)\oplus M(2^i)=H(2^i)$. These sets will be used to
construct $Q(n)$.

\begin{lemma}
    For any integer $0\leq m<2^{k-1}$,
    \[
    H(m2^{n+1})M(2^n)H((2^k-2m-1)2^n)\subseteq U(2^{n+k}).
    \]
\end{lemma}

\begin{proof}
    By definition, $M(2^n)H(2^n) \subseteq U(2^{n+1})$. Using this fact
    and
    Lemma~\ref{Ustack},
    \begin{eqnarray*}
    \lefteqn{H(m2^{n+1})M(2^n)H((2^k-2m-1)2^n) \subseteq}\\
    &\subseteq& H(m2^{n+1})U(2^{n+1})H((2^{k-1}-m-1)2^{n+1})
    \subseteq U(2^{n+k}),
    \end{eqnarray*}
    as required.
\end{proof}

\begin{theorem}\label{Qsize}
    Let $j \in \mathbb{N}$. Write $j$ in binary form as
        $$j=2^{p_0}+2^{p_1}+...+2^{p_n}$$
    with $0\leq p_0<p_1<...<p_n$, and suppose $n \neq 0$
    (that is, $j$ is not a power of $2$).
    Then there is linear space $Q(j)\subseteq H(j)$ such that
    $Q(j)\oplus R(j)=H(j)$ and
        $$Q(j)\subseteq N(2^{p_n})N(2^{p_{n-1}})...N(2^{p_0})
        =\displaystyle\prod_{i=0}^n N(2^{p_{n-i}})
        \subseteq\displaystyle\prod_{i=0}^n V(2^{p_{n-i}})$$
\end{theorem}

\begin{proof}
    This proof is very similar to the one for Theorem~\ref{Wsize}. By
    definition, $H(2^{p_i})=N(2^{p_i})\oplus M(2^{p_i})$ for $i=1,2,...,n$.
    Hence, $H(j)\subseteq \prod_{i=0}^{n}(N(2^{p_i})\oplus M(2^{p_i}))$, and
    $$H(j) = \left(\displaystyle\sum_{i=0}^n
    H(2^{p_n}+...+2^{p_{i+1}})M(2^{p_i})H(2^{p_{i-1}}+...+2^{p_0}) \right)
    \oplus \displaystyle\prod_{i=0}^n N(2^{p_i})$$
    Set $B_{p_i}(j)
    :=H(2^{p_n}+...+2^{p_{i+1}})M(2^{p_i})H(2^{p_{i-1}}+...+2^{p_0})$,
    so that
    $H(j) = \sum_{i=0}^n B_{p_i}(j) \oplus \prod_{i=0}^n N(2^{p_i})$.

    Multiplying on the right by $H(2^{p_n+1}-j)$ we obtain
    \begin{eqnarray*}
    B_{p_i}(j)H(2^{p_n+1}-j)
    &=&
    H(2^{p_n}+...+2^{p_{i+1}})M(2^{p_i})H(2^{p_n+1}-(2^{p_n}+...+2^{p_i}))\\
    &=&
    H\left((2^{p_n-p_i-1}+...+2^{p_{i+1}-p_i-1})2^{p_i+1}\right)
    M(2^{p_i})H(2^{p_i})\times \\
    &&H\left((2^{p_n-p_i}-(2^{p_n-p_i-1}+...+2^{p_{i+1}-p_i-1}-1))
    2^{p_i+1}\right) \\
    &\subseteq&
    H\left((2^{p_n-p_i-1}+...+2^{p_{i+1}-p_i-1})2^{p_i+1}\right)
    U(2^{p_i+1})\times \\
    &&H\left((2^{p_n-p_i}-(2^{p_n-p_i-1}+...+2^{p_{i+1}-p_i-1}-1))
    2^{p_i+1}\right)
    \end{eqnarray*}
    It follows from Lemma~\ref{Ustack} that
    $B_{p_i}(j)H(2^{p_n+1}-j) \subseteq U(2^{p_n+1})$;
    so that each $B_{p_i}(j) \subseteq R(j)$.
    By exactly the same reasoning
    as in Theorem~\ref{Wsize}, there must exist some
    $Q(j) \subseteq \prod_{i=0}^n N(2^{p_{n-i}})$
    such that $R(j) \oplus Q(j) = H(j)$.
\end{proof}

One last theorem about the size of $Q$ and $W$ must
be obtained  before continuing.
\begin{theorem}\label{Qadd}
    Suppose that $j, k \in \mathbb{N}$ have binary forms
        $$k = 2^{p_0} + ... + 2^{p_{i-1}}\qquad
        j = 2^{p_i} + ... + 2^{p_n}.$$
    with $p_0 < ... < p_n$.  Then $\dim Q(j+k) \leq \dim Q(j) \dim Q(k)$ and
    $\dim W(j+k) \leq \dim W(j) \dim W(k)$.
\end{theorem}
\begin{proof}
    Use the definition of $Q$ to see that:
        $$H(j+k) = (R(j) \oplus Q(j))(R(k) \oplus Q(k)) = R(j)H(k) \oplus Q(j)R(k) \oplus Q(j)Q(k)$$
    If it can be shown that $R(j)H(k) + Q(j)R(k) \subseteq R(j+k)$, then
    \begin{eqnarray*}
    \lefteqn{\dim Q(j+k) = \dim H(j+k) - \dim R(j+k) \leq}\\
        &&\dim H(j) \dim H(k) - \dim R(j) \dim H(k) - \dim Q(j) \dim R(k)\\
        &=&
        \dim Q(j) \dim H(k) - \dim Q(j) \dim R(k) = \dim Q(j) \dim Q(k),
    \end{eqnarray*}
    which establishes the $Q$ inequality.

    In order to show that $R(j)H(k) \subseteq R(j+k)$, note that $2^{p_n} <
    j+k < 2^{p_n+1}$, and recall from the definition of $R(j)$ that
    $$R(j)H(k)
    \cdot H(2^{p_n+1}-j-k) = R(j)H(2^{p_n+1}-j) \subseteq U(2^{p_n+1})$$
    so that $R(j)H(k) \subseteq R(j+k)$ by the definition of $R(j+k)$.

    Finally, to show that $Q(j)R(k) \subseteq R(j+k)$, note that $2^{p_{i-1}} \leq
    k < 2^{p_{i-1}+1}$ and
    \begin{eqnarray*}
    Q(j)R(k)H(2^{p_n+1}-j-k)
    &=&
    Q(j)\left(R(k)H(2^{p_{i-1}+1}-k)\right)H(2^{p_n+1}-2^{p_{i-1}+1}-j)\\
    &\subseteq&
    H(j)U(2^{p_{i-1}+1})H(2^{p_n+1}-2^{p_{i-1}+1}-j).
    \end{eqnarray*}
    As  each of $j$ and
    $2^{p_n+1}-2^{p_{i-1}+1}-j$ is divisible by $2^{p_{i-1}+1}$,
    Theorem~\ref{Ustack} reveals that $Q(j)R(k)H(2^{p_n+1}-j-k) \subseteq
    H(2^{p_n+1})$ and so $Q(j)R(k) \subseteq R(j+k)$.

    An analogous argument is used to prove the inequality for $W$.
\end{proof}



\section{Inequalities}

In this section we will prove, using induction, that for all $n>1$,
    $$\dim Q(n), \dim W(n)\leq 8\sqrt{n} ( \log n )^3$$

This result is obtained by combining the following theorems.

\begin{theorem}\label{WQsmall}
    If $2^m < n < 2^{m+1}$, then
        $$\dim W(n) \leq \dim Q(2^{m+1}-n) \dim V(2^{m+1})$$
    and
        $$\dim Q(n) \leq \dim W(2^{m+1}-n) \dim V(2^{m+1})$$
\end{theorem}
\begin{proof}
    By the definition, $R(2^{m+1}-n)H(n)\subseteq U(2^{m+1})$.  Therefore, if
    $c\in H(n)$ and $Q(2^{m+1}-n)c\subseteq U(2^{m+1})$, then
    $H(2^{m+1}-n)c\subseteq U(2^{m+1})$ and $c\in S(n)$.

    Let $v_1,...,v_d\in
    Q(2^{m+1}-n)$ be a basis of $Q(2^{m+1}-n)$ over $K$ and
    let $c_1, ... ,c_p\in W(n)$ be a basis of $W(n)$. Suppose that
    $p =\dim(W(n))> \dim Q(2^{m+1}-n)\dim V(2^{m+1}) =d\dim V(2^{m+1})$.

    Define a $K$-linear function
    $f:W(n)\rightarrow (H(2^{m+1})/U(2^{m+1}))^d \cong V(2^{m+1})^d$ by
    setting
    \[f(c):=
    \left((v_1 c+U(2^{m+1})), (v_2 c+U(2^{m+1}),...,(v_d c+U(2^{m+1}))
    \right)
       \]
    for each $c\in W(n)$.

    Observe that $\dim (\text{Im} f) \leq \dim (H(2^{m+1})/U(2^{m+1}))^d = d
    \dim V(2^{m+1})$, and that since $\dim W(n) = p > d \dim V(2^{m+1})$,
    there must exist some non-zero $c \in \ker f$. However, if $(v_i c +
    U(2^{m+1})) = 0$ for each $v_i$, then $Q(2^{m+1}-n) c \in U(2^{m+1})$ and
    $c \in S(n)$. Hence, $c\in S(n) \cap W(n) = \{0\}$, a contradiction. Thus,
    $\dim W(n) = p \leq \dim Q(2^{m+1}-n) \dim V(2^{m+1})$, as required.

    The second inequality can be proven by a similar argument.
\end{proof}

\begin{theorem}\label{Sdim}
    Let $j$ be a natural number. Write $j$ in binary form as
        $$j=2^{p_{0}}+2^{p_{1}}+\ldots +2^{p_{n}}$$
    with $0\leq p_{0}<p_{1}< ... <p_{n}$.
    Recalling the sets $\{S_i\}_{i \in Z}$ from Section~\ref{defuv},
    suppose that there is an $m\in Z$ with $p_0, ..., p_n \in S_m$.
    Then $\dim Q(j)\leq 2\sqrt{j} \lfloor \log j\rfloor$ and
    $\dim W(j)\leq 2\sqrt{j} \lfloor \log j\rfloor$.
\end{theorem}

\begin{proof}
    This proof will be divided into three cases:
    \begin{itemize}
        \item[1.] $j < 2^{2^m-\lfloor \log m\rfloor-1}$
        \item[2.] $j = 2^{2^m-\lfloor \log m\rfloor-1}$
        \item[3.] $j > 2^{2^m-\lfloor \log m\rfloor-1}$
    \end{itemize}

    {\bf Case 1.} Suppose that $j<2^{2^m-\lfloor \log m\rfloor-1}$. Then
    $p_n<2^m-\lfloor log(m)\rfloor-1$.
    Notice that $j \geq 2^{2^m - \lfloor \log
    m\rfloor - m}$ and
    \begin{eqnarray*}
    \sqrt{j} \lfloor \log j \rfloor &\geq& 2^{2^{m-1} -
    \lfloor \log m\rfloor/2 - m/2} (2^m - \lfloor \log m\rfloor - m)\\
    &>&2^{2^{m-1} - \lfloor \log m\rfloor/2 - m/2} 2^{m-1} > 2^{2^{m-1}}
    \end{eqnarray*}
    Hence, by using Theorem~\ref{Qsize}, we obtain
    \[
    \dim Q(j)\leq
    \dim\prod_{i=0}^n V(2^{p_i})\leq \prod_{i=0}^{m-2} 2^{2^i} < 2^{2^{m-1}} <
    \sqrt{j} \lfloor \log j \rfloor,
    \]
    as required.

    A similar argument, using Theorem~\ref{Wsize}, gives
    $\dim W(j)\leq \sqrt{j} \lfloor \log j \rfloor$.

    {\bf Case 2.} Suppose that $j = 2^{2^m-\lfloor \log m\rfloor-1}$. Then, by
    definition, $U(j)\subseteq R(j)\cap S(j)$, and so $\dim Q(j), \dim W(j)
    \leq \dim V(2^{2^m-\lfloor \log m\rfloor-1}) = 2^{2^{m-1}}$, by
    Theorem~\ref{8props}(2). Consequently, $\dim Q(j), \dim W(j)\leq \sqrt{j}
    \lfloor \log j\rfloor.$

    {\bf Case 3.} Suppose that $j>2^{2^m-\lfloor \log m\rfloor-1}$. Then $p_n
    = 2^m-\lfloor \log m\rfloor-1$ and $2^{p_n+1}-j< 2^{2^m-\lfloor \log
    m\rfloor-1}$.  Set $k:= 2^{p_n+1}-j$, and note that $k<j$ and that
    Case 1 applies to $k$. Thus, an application of Case 1 gives
    $\dim
    Q(k), \dim W(k) \leq \sqrt{k} \lfloor \log k \rfloor
    <\sqrt{j} \lfloor \log j \rfloor$.

    Now, apply Theorem~\ref{WQsmall} to see that
    \begin{eqnarray*}
    \lefteqn{\dim Q(j) \leq \dim W(2^{p_n+1}-j)
    \dim V(2^{p_n+1}) =}\\&&\dim W(k) \cdot 2 \leq 2\sqrt{k} \lfloor \log k
    \rfloor < 2\sqrt{j} \lfloor \log j \rfloor,
    \end{eqnarray*}
    as required.

    A similar argument shows that
    $\dim W(j)\leq 2\sqrt{j} \lfloor \log j \rfloor$ in this case.

    This finishes the three cases and thus also the proof.
\end{proof}

Now, for each $m \in Z$, define $T_m \subset \mathbb{N}$ to be the set bounded
above by $S_m$ and below by $S_{m'}$, where $m'$ is the next lowest value in
$Z$ (or by 0, if $m$ is the lowest value of $Z$). More formally, if $m', m \in
Z$ with $m'<m$ and $(m', m)\cap Z = \emptyset$, then set $T_m = [2^{m'} -
\lfloor \log m' \rfloor, 2^m - m - \lfloor \log m \rfloor - 1]$. If $m$ is the
minimal value of $Z$, then set $T_m = [1, 2^m - m - \lfloor \log m \rfloor -
1]$. The subsets $\{S_m, T_m\}_{m\in
Z}$ provide a partition of $\mathbb{N}$.

\begin{theorem}\label{Tdim}
    Let $j$ be a natural number. Write $j$ in binary form as
    $$j=2^{p_0}+2^{p_1}+...+2^{p_n}$$ with $0\leq p_0 < p_1 < ... <p_n$. If
    there exists an $m\in Z$ such that $p_0, ..., p_n \in T_m$, then
    $\dim Q(j), \dim W(j)\leq 2$.
\end{theorem}

\begin{proof}
    Note that
    $\dim Q(j)\leq \dim\left(\prod_{i=0}^n N(2^{p_i})\right) = 1$,
    by Theorem~\ref{Qsize}, because
    $p_0, ..., p_n \notin S$.

    For the $W(j)$ case, note that $2^{p_n} \leq j < 2^{p_n+1}$, and let
    $2^{q_0} + ... + 2^{q_n}$ be the binary form of $2^{p_n+1}-j$. As $p_0
    = q_0 < ... < q_n < p_n$, it follows that $q_0, ..., q_n \in T_m$
    and so
    $q_0, ..., q_n \notin S$. Applying Theorem~\ref{Qsize} in this case
    gives $\dim
    Q(2^{p_n+1}-j) \leq 1$, and then applying Theorem~\ref{WQsmall}
    gives $$W(j) \leq
    Q(2^{p_n+1}-j)V(2^{p_n+1}) \leq 2,$$
    as required.
\end{proof}

We can now establish the main estimate of this section.

\begin{theorem}
    For each $n > 1$,
    \[
    \dim Q(n), \dim W(n)\leq 8 \sqrt{n} (\log
    n)^3.
    \]
\end{theorem}

\begin{proof}
    Let $n = 2^{p_0} + 2^{p_1} + ...$ be the binary decomposition of
    $n$. For each $m \in Z$, let $j_m$ be the sum of each $2^{p_i}$
    that occurs in the binary form of $n$ with $p_i \in S_m$, and let
    $k_m$ be the sum of each $2^{p_i}$ with $p_i \in T_m$. Then 
    
$$n = \sum_{\substack{m\in Z\\j_m\neq 0}}
        j_m 
        + \sum_{\substack{m\in Z\\k_m\neq 0}}k_m ,$$
        as $\{S_m, T_m\}_{m\in Z}$
    forms a partition of $\mathbb{N}$.

Therefore, 
$$\dim Q(n)\leq \prod_{\substack{m\in Z\\j_m\neq 0}}\dim Q(j_{m})
        \prod_{\substack{m\in Z\\k_m\neq 0}}\dim Q(k_{m}),$$ 
        by Theorem~\ref{Qadd}. 
        
We estimate the two terms on the right hand side separately. 

    First, suppose that $m<r$ are consecutive members of $Z$ with $k_r\neq 0$.
    Then $2^{2^m -\lfloor \log m\rfloor}\leq k_r\leq n$, as $T_r=[2^{m} -
    \lfloor \log m \rfloor, 2^r - r - \lfloor \log r \rfloor - 1]$. It follows
    that $m\leq \log\log(n) +1$ in this case. Therefore, the number of 
    $m\in Z$ with   $k_m\neq 0$ is 
    $\leq \log\log(n) +2$. Note that $\dim Q(k_{m})\leq 2$, for each
    such $k_{m}$, by Theorem~\ref{Tdim}, so that 
$$\prod_{\substack{m\in Z\\k_m\neq 0}}\dim Q(k_{m})\leq 
\prod_{i=1}^{\lfloor\log\log(n)+2\rfloor}2
        \leq 2^{\log\log(n)+2} \leq 4\log n .$$

Secondly, observe that if $j_m\neq 0$ 
then $2^{2^m-m-\lfloor \log m\rfloor}\leq
j_m \leq n$ and $ j_m <2^{2^m}$, because $S_m=[2^m-m-\lfloor \log m\rfloor,
2^m-\lfloor \log m\rfloor-1]$. Also, observe that $\dim Q(j_m)\leq
2\sqrt{j_m}\lfloor\log j_m\rfloor \leq j_m$, by Theorem~\ref{Sdim}.

Suppose that $t<r$ are consecutive members of $Z$ such that 
$r$ is the largest member of $Z$ such that $j_r\neq 0$.
Note
that $2^{2^{2^{2t}}}< r$, by Lemma~\ref{enumerate}; so that $2^{2^{2t}}<\log
r\leq \log n$. 

Consider any $m\in Z$ with $m\leq t$ and $j_m\neq 0$. 
Any $p_i$ involved in the sum 
$j_m$ satisfies 
\[
p_i\leq 2^m-\lfloor\log m\rfloor -1\leq 2^m\leq 2^t.
\]
As each $p_i$ can be involved in at most one such sum $j_m$, the number of 
$m\in Z$ with $j_m\neq 0$ is $\leq 2^t$. For any such $m$, 
\[
\dim Q(j_m) \leq j_m\leq 2^{2^m}\leq 2^{2^t}
\]
Thus, 
\[
\prod_{\substack{m\in Z\\k_m\neq 0\\m<r}}\dim Q(k_{m})
\leq (2^{2^t})^{2^t} = 2^{2^{2t}} \leq \log r\leq \log n
\]
Hence, 
\[
        \prod_{\substack{m\in Z\\k_m\neq 0}}\dim Q(k_{m})
        \leq (\log n)\cdot (2\sqrt{j_r}\log j_r)
        \leq 2\sqrt{n}(\log n)^2
        \]
        and so
\begin{eqnarray*}
 \dim Q(n)&\leq& \prod_{\substack{m\in Z\\j_m\neq 0}}\dim Q(j_{m})
        \prod_{\substack{m\in Z\\k_m\neq 0}}\dim Q(k_{m})\\[1ex]
        &\leq& (4\log n) \cdot     (2\sqrt{n}(\log n)^2) 
        = 8  \sqrt{n}(\log n)^3,
        \end{eqnarray*} 
as required.

    To show that
    $W(n)\leq  8\sqrt {n}(\log n )^{3}$ we use
    an analogous argument.
\end{proof}

Now we are ready to obtain the main result of the paper.
\begin{theorem} The algebra $\bar{E}/A$ is a finitely generated infinite
dimensional nil algebra with Gelfand-Kirillov dimension at most $3$.
\end{theorem}

\begin{proof}
The algebra $\bar{E}/A$ is a finitely generated infinite
dimensional nil algebra, by Theorem~\ref{AbaroverEnil}.

    By combining the previous theorem with Theorem~\ref{totalsize},
    we obtain
    \begin{eqnarray*}
    \dim H(n)/E(n)&\leq& \sum_{k=0}^n \dim(W(n-k))\dim(Q(k))\\
    &\leq&
    \sum_{k=0}^n 64\sqrt{(n-k)k\;} (\log (n-k) \log k)^3
        < 64 n^2 (\log n)^6
    \end{eqnarray*}
Hence,
    \[
    \sum_{i=1}^n \dim H(i)/E(i) \leq 64 n^3 (\log n)^6.
    \]
Therefore,
\begin{eqnarray*}
\GKdim(\bar{A}/E)&=&
\overline{\lim_{n\rightarrow\infty}} \left(
        \frac{\log (\sum_{i=1}^n\dim H(i)/E(i))}{\log n}\right)\\
    &\leq&
    \overline{\lim_{n\rightarrow\infty}} \left(
        \frac{6+3(\log n) + 6(\log \log n)}{\log n}\right) = 3,
   \end{eqnarray*}
   as required.
\end{proof}



\section*{Concluding remarks and some questions}

We have constructed a finitely generated infinite dimensional nil algebra with
Gelfand-Kirillov dimension at most three. Equivalently, we have a finitely
generated infinite dimensional nil but not nilpotent algebra with
Gelfand-Kirillov dimension at most three.

In contrast, nil does imply nilpotent for algebras of Gelfand-Kirillov
dimension at most one, by \cite{ssw}. Combining this with Bergman's Gap
Theorem, \cite[Theorem 2.5]{kl}, we see that a nil but not nilpotent example
must have Gelfand-Kirillov dimension at least two. It would be very
interesting to find the precise dividing line in terms of growth. A starting
point might be to consider nil algebras with quadratically bounded growth and
attempt to show that these algebras must be finite dimensional. Given a
positive result in this direction, one might then speculate whether there
exists a finitely generated nil but not nilpotent algebra with
Gelfand-Kirillov dimension two (but, of course, not having quadratic growth).

Many of the constructions of weird algebras that we know involve starting with
a free algebra and introducing infinitely many relations; so the corresponding questions for finitely presented algebras remain unresolved.
In particular, we ask: is every finitely presented nil algebra nilpotent?

It seems unlikely that by using the methods employed in this work we
can hope to construct a nil but not nilpotent algebra with Gelfand-Kirillov
dimension two. Our algebras are graded, and this raises the question of
whether a finitely generated nil algebra that is graded and has
Gelfand-Kirillov dimension at most two (or quadratic growth) must in fact be
finite dimensional.

The methods employed here depend crucially on the countability hypothesis. It
would be interesting to see if it is possible to construct a finitely
generated infinite dimensional nil algebra with finite Gelfand-Kirillov
dimension over an uncountable field.

There are many problems of a similar type in Zelmanov's paper \cite{zelmanov}.


\vspace{3ex}

\begin{minipage}{1.00\linewidth}
\noindent
T H Lenagan, Agata Smoktunowicz: \\
Maxwell Institute for Mathematical Sciences\\
School of Mathematics, University of Edinburgh,\\
James Clerk Maxwell Building, King's Buildings, Mayfield Road,\\
Edinburgh EH9 3JZ, Scotland, UK\\
E-mail: {\tt tom@maths.ed.ac.uk}, {\tt A.Smoktunowicz@ed.ac.uk}
\\

\noindent Alexander Young:\\
Department of Mathematics\\
University of California at San Diego\\
9500 Gilman Drive\\
La Jolla, CA 92093-0112\\
USA\\
E-mail: {\tt aayoung@math.ucsd.edu}
\end{minipage}

\end{document}